\documentclass[psamsfonts,fceqn,leqno]{amsart}
\usepackage{mathrsfs,latexsym,amsfonts,amssymb,curves,epic}
\usepackage{color}

\newtheorem{theorem}{Theorem}[section]
\newtheorem{corollary}[theorem]{Corollary}
\newtheorem{proposition}[theorem]{Proposition}
\newtheorem{lemma}[theorem]{Lemma}
\newtheorem{question}[theorem]{Question}

\theoremstyle{definition}

\newtheorem{remark}[theorem]{Remark}
\newtheorem{example}[theorem]{Example}

\newcommand{\uhr}{\upharpoonright}
\def\N{{\mathbb N}}

\def\R{{\mathbb R}}
\def\Z{{\mathbb Z}}

\begin{document}
\title[Two Weak Forms of Countability Axioms in Free Topological Groups]
{Two Weak Forms of Countability Axioms\\ in Free Topological Groups}

  \author{Fucai Lin}
  \address{(Fucai Lin): School of mathematics and statistics,
  Minnan Normal University, Zhangzhou 363000, P. R. China}
  \email{linfucai2008@aliyun.com}

  \author{Chuan Liu}
  \address{(Chuan Liu): Department of Mathematics,
  Ohio University Zanesville Campus, Zanesville, OH 43701, USA}
  \email{liuc1@ohio.edu}

  \author{Jiling Cao}
  \address{(Jiling Cao): School of Engineering, Computer and
  Mathematical Sciences, Auckland University of Technology,
  Private Bag 92006, Auckland 1142, New Zealand}
  \email{jiling.cao@aut.ac.nz}

  \thanks{The first author is supported by the NSFC (Nos. 11571158, 11201414, 11471153),
  the Natural Science Foundation of Fujian Province (No. 2012J05013) of China
  and Training Programme Foundation for Excellent Youth Researching Talents
  of Fujian's Universities (JA13190) and the project of Abroad Fund of Minnan
  Normal University. The third author acknowledges the support and hospitality
  of Minnan Normal University, when he was appointed a Min Jiang
  Scholar and Visitng Professor from 2014 to 2016. This paper was partially
  written when the first author was visiting the School of Computer and
  Mathematical Sciences at Auckland University of Technology from March
  to September 2015, and he wishes to thank the hospitality of his host.}

  %insert keywords
  %insert subject class

  %\date{\today}
  \begin{abstract}
  Given a Tychonoff space $X$, let $F(X)$ and $A(X)$ be respectively the free
  topological group and the free Abelian topological group over $X$ in the sense
  of Markov. For every $n\in\mathbb{N}$, let $F_{n}(X)$ (resp. $A_n(X)$) denote
  the subspace of $F(X)$ (resp. $A(X)$) that consists of words of reduced length
  at most $n$ with respect to the free basis $X$. In this paper, we discuss two
  weak forms of countability axioms in $F(X)$ or $A(X)$, namely the
  $csf$-countability and $snf$-countability. We provide some characterizations
  of the $csf$-countability and $snf$-countability of $F(X)$ and $A(X)$ for
  various classes of spaces $X$. In addition, we also study the $csf$-countability
  and $snf$-countability of $F_n(X)$ or $A_n(X)$, for $n=2, 3, 4$. Some results of
  Arhangel'ski\v\i\ in \cite{A1980} and Yamada in \cite{Y1998} are generalized.
  An affirmative answer to an open question posed by Li et al. in \cite{LLL} is
  provided.
  \end{abstract}

  \maketitle

  \section{Introduction}
  In 1941, Markov \cite{MA1945} introduced the concepts of the free topological
  group $F(X)$ and the free Abelian topological group $A(X)$ over a Tychonoff
  space $X$, respectively. Since then, free topological groups have been a
  source of various examples and also an interesting topic of study in the
  theory of topological groups, see \cite{AT2008}. From the algebraic point
  of view, the structure of $F(X)$ or $A(X)$ is very simple - it is the free
  algebraic group over the set $X$. However, the topological structure of
  $F(X)$ and $A(X)$ is rather complicated even for simple spaces $X$. For
  example, it is a well known fact that if $X$ is a non-discrete space,
  then neither $F(X)$ nor $A(X)$ is Fr\'echet-Uryshon, and hence first
  countable, see \cite{A1980}. This fact motivates researchers to investigate
  free topological groups in two directions. The first direction of the research
  on free topological groups is to study some weak forms of countability
  axioms in $F(X)$ and $A(X)$ over certain classes of spaces $X$. In this
  line, Arhangel'ski\v\i\ et al. \cite{AOP1989} considered the following
  questions on $F(X)$ and $A(X)$ over a metrizable space $X$: \emph{For
  which spaces $X$, is $F(X)$ or $A(X)$ a $k$-space? When is the tightness
  of $F(X)$ or $A(X)$ countable?} They proved that $F(X)$ is a $k$-space
  iff $X$ is locally compact separable or discrete; $A(X)$ is a $k$-space
  iff $X$ is locally compact and $X'$ is separable, where $X'$ is the
  derived set of $X$. Furthermore, the tightness of $F(X)$ is countable
  iff $X$ is separable or discrete, and the tightness of $A(X)$ is
  countable iff $X'$ is separable.

  %\medskip
  The other direction of research on free topological groups is to study
  (weak) countability axioms of $F_n(X)$ or $A_n(X)$, where $F_n(X)$
  (resp. $A_n(X)$) stands for the subset of $F(X)$ (resp. $A(X)$) formed
  by all words whose reduced length is at most $n$. Indeed, Yamada
  \cite{Y2002} showed that for a metrizable space $X$, $F_3(X)$ or
  $A_3(X)$ is Fr\'echet-Uryshon iff $X'$ is compact, and $F_5(X)$ is
  Fr\'echet-Uryshon iff $X$ is compact or discrete. As applications,
  characterizations of a metrizable space $X$ are given such that
  $A_n(X)$ is Fr\'echet-Uryshon for each $n\ge 3$, and $F_n(X)$ is
  Fr\'echet-Uryshon for each $n\ge 3$ except for $n=4$. The subspaces
  $F_4(X)$ and $A_4(X)$ are very special cases. In \cite{Y1998}, Yamada
  proved that for a metrizable space $X$, the following are equivalent:
  (i) $F_n(X)$ is metrizable for each $n\in \N$; (ii) $F_n(X)$ is first
  countable for each $n\in \N$; (iii) $F_4(X)$ is metrizable; (iv)
  $F_4(X)$ is first countable; (v) $X$ is compact or discrete. In the
  same paper, Yamada also studied the first countability of $F_n(X)$
  and $A_n(X)$ for $n=2,3$. It is proved that for a metrizable space
  $X$, the following are equivalent: (i) $F_{3}(X)$ is metrizable;
  (ii) $F_{3}(X)$ is first-countable; (iii) $F_{2}(X)$ is metrizable;
  (iv) $F_{2}(X)$ is first-countable; (v) $X'$ is compact. Furthermore,
  for a metrizable space $X$, the following are also equivalent: (i)
  $A_{2}(X)$ is first-countable; (ii) $A_{2}(X)$ is metrizable; (iii)
  $A_{n}(X)$ is first-countable for each $n\in\mathbb{N}$; (iv)
  $A_{n}(X)$ is metrizable for each $n\in\mathbb{N}$; (v) $X'$ is
  compact.

  %\medskip
  Recently, Li et al. \cite{LLL} continued the study of $F(X)$ and $A(X)$
  along the afore-mentioned first direction. They studied several weak
  forms of countability axioms of $F(X)$ and $A(X)$ defined by networks
  over some classes of generalized metric spaces $X$. More precisely,
  they studied the concepts of $sn$-networks, $cs$-networks, $cs^*$-networks
  in $F(X)$, $A(X)$, and their subspaces $F_n(X)$ and $A_n(X)$. Two types
  of countability axioms defined by these concepts, namely $snf$-countability
  and $csf$-countability,
  were considered. Among many other things, Li et al. established the
  following results: For a metrizable and crowded space $X$, $F(X)$ or
  $A(X)$ is $csf$-countable iff $X$ is separable; For a stratifiable
  $k$-space $X$, $F(X)$ is $snf$-countable iff $X$ is discrete.
  However, the authors of \cite{LLL} did not consider the
  $snf$-countability and $csf$-countability of $F_n(X)$ and $A_n(X)$.

  %\medskip
  In the paper, we continue the study of free topological group $F(X)$
  and the free Abelian topological group $A(X)$ in the afore-mentioned
  two directions. In particular, we investigate the $csf$-countability and
  the $snf$-countability of $F(X)$, $A(X)$, $F_n(X)$ and $A_n(X)$ over
  various classes of generalized metric spaces $X$. In Section 2,
  we introduce the necessary notation and terminologies which are
  used for the rest of the paper. In Section 3, we investigate the
  $snf$-countability of free (Abelian) topological groups. First, we
  provide some characterizations of the $snf$-countability of $F(X)$,
  $A(X)$, $F_n(X)$ and $A_n(X)$ over certain classes of topological
  spaces. The main theorem in this section generalizes a result of
  Yamada in \cite{Y1998}. Section 4 is devoted to the study of the
  $csf$-countability of $F(X)$, $A(X)$, $F_n(X)$ and $A_n(X)$. It is
  shown that for a non-discrete La\v{s}nev space $X$, $F_4(X)$ is
  $csf$-countable iff $F_4(X)$ is an $\aleph_0$-space. This result
  gives an affirmative answer to an open question in \cite{LLL}. It
  is also shown that for a sequential $\mu$-space $X$, if $X'$
  has a countable $cs^*$-network in $X$, then $F_3(X)$ and $A(X)$
  are $csf$-countable. Finally, we pose several interesting open
  questions in the last section.

  %\medskip
  Throughout this paper, all topological spaces are assumed to be
  at least Tychonoff, unless explicitly stated otherwise.

  \section{Notation and Terminologies}

  In this section, we introduce the necessary notation and terminologies.
  First of all, let $\N$, $\Z$ and $\R$ denote the sets of all positive
  integers, all integers and all real numbers, respectively. For undefined
  terminologies, the reader may refer to \cite{AT2008}, \cite{E1989} and
  \cite{Gr}.

  %\medskip
  Let $X$ be a topological space $X$ and $A \subseteq X$ be a subset of $X$.
  The \emph{closure} of $A$ in $X$ is denoted by $\overline{A}$ and the
  \emph{diagonal} of $X$ is denoted by $\Delta (X)$. The subset $A$ is called
  \emph{$C^{\ast}$-embedded in $X$} if every bounded continuous real-valued
  function $f$ defined on $A$ has a bounded continuous extension over $X$.
  Moreover, $A$ is called \emph{bounded} if every continuous real-valued
  function $f$ defined on $A$ is bounded. If the closure of every bounded
  set in $X$ is compact, then $X$ is called a \emph{$\mu$-space}. The
  \emph{derived set}  of $X$ is denoted by $X'$. We say that $X$ is
  \emph{crowded} if $X = X'$. Recall that $X$ is said to have a
  \emph{$G_{\delta}$-diagonal} (resp. \emph{regular $G_\delta$-diagonal})
  if $\Delta(X)$ is a $G_\delta$-set (resp. regular $G_\delta$-set) in $X
  \times X$. A pseudometric $d$ on $X$ is said to be \emph{continuous} if
  $d$ is continuous as a mapping from the product space $X\times X$ to
  $\mathbb{R}$. The space $X$ is called a \emph{$k$-space} provided that a
  subset $C\subseteq X$ is closed in $X$ if $C\cap K$ is closed in $K$ for
  each compact subset $K$ of $X$. If there exists a family of countably many
  compact subsets $\{K_{n}: n\in\mathbb{N}\}$ of $X$ such that each subset
  $F$ of $X$ is closed in $X$ provided that $F\cap K_n$ is closed in $K_n$
  for each $n\in\mathbb{N}$, then $X$ is called a \emph{$k_\omega$-space}.
  Note that every $k_{\omega}$-space is a $k$-space. In addition, $X$ is
  called a \emph{$cf$-space} if every compact subset of $X$ is finite. A
  subset $P$ of $X$ is called a \emph{sequential neighborhood} \cite{FR1965}
  of $x \in X$, if each sequence converging to $x$ is eventually in $P$.
  A subset $U$ of $X$ is called \emph{sequentially open} if $U$ is a
  sequential neighborhood of each of its points. The space $X$ is called
  a \emph{sequential space} if each sequentially open subset of $X$ is
  open. Let $\kappa$ be an infinite cardinal. For each $\alpha\in\kappa$,
  let $T_{\alpha}$ be a sequence converging to $x_\alpha\not\in T_\alpha$.
  Let $T:=\bigoplus_{\alpha\in\kappa}(T_\alpha\cup\{x_\alpha\})$ be the
  topological sum of $\{T_{\alpha}\cup \{x_{\alpha}\}: \alpha\in\kappa\}$.
  Then $S_{\kappa} :=\{x\}  \cup \bigcup_{\alpha\in\kappa}T_{\alpha}$ is
  the quotient space obtained from $T$ by identifying all the points
  $x_{\alpha}\in T$ to the point $x$.

  %\medskip
  Let $\mathscr P$ be a family of subsets of $X$. Then, $\mathscr P$ is called
  a \emph{$cs$-network} \cite{G1971} at a point $x\in X$ if for every sequence
  $\{x_n: n \in \N\}$ converging to $x$ and an arbitrary open neighborhood
  $U$ of $x$ in $X$ there exist an $m\in\N$ and an element $P\in
  \mathscr P$ such that
  \[
  \{x\}\cup \{x_n: n\geqslant m\} \subseteq P \subseteq U.
  \]
  The space $X$ is called \emph{csf-countable} if $X$ has a countable $cs$-network
  at each point $x\in X$. We call $\mathscr P$ a \emph{$cs^{\ast}$-network}
  \cite{LT1994} of $X$ if for every sequence $\{x_n: n\in\N\}$ converging to a
  point $x$ and an arbitrary open neighborhood $U$ of $x$ in $X$ there
  is an element $P\in\mathscr P$ and a subsequence $\{x_{n_{i}}: i\in \N\}$ of
  $\{x_n: n\in \N\}$ such that $\{x\}\cup\{x_{n_{i}}: i\in \mathbb{N}\}
  \subseteq P\subseteq U$. Furthermore, $\mathscr P$ is called a {\it $k$-network}
  \cite{O1971} if for every compact subset $K$ of $X$ and an arbitrary open set
  $U$ containing $K$ in $X$ there is a finite subfamily $\mathscr {P}^{\prime}
  \subseteq \mathscr {P}$ such that $K\subseteq \bigcup\mathscr {P}^{\prime}
  \subseteq U$. Recall that a space $X$ is an \emph{$\aleph$-space} (resp.
  {\it $\aleph_{0}$-space}) if $X$ has a $\sigma$-locally finite (resp.
  countable) $k$-network. Let $\mathscr P$ be a cover of $X$ such that (i)
  $\mathscr P = \bigcup_{x\in X}\mathscr{P}_{x}$; (ii) for each point $x\in X$,
  if $U,V\in \mathscr{P}_{x}$, then $W\subseteq U\cap V$ for some $W\in
  \mathscr{P}_{x}$; and (iii) for each point $x\in X$ and each open neighborhood
  $U$ of $x$ there is some $P\in\mathscr P_x$ such that $x\in P \subseteq U$.
  Then, $\mathscr P$ is called an \emph{sn-network} \cite{Lin1996} for $X$ if
  for each point $x\in X$, each element of $\mathscr P_x$ is a sequential
  neighborhood of $x$ in $X$, and $X$ is called \emph{snf-countable}
  \cite{Lin1996} if $X$ has an $sn$-network $\mathscr P$ and $\mathscr P_x$
  is countable for all $x\in X$. The following implications follow directly
  from definitions:
  \[
  \mbox{first countable} \Rightarrow \mbox{$snf$-countable} \Rightarrow
  \mbox{$csf$-countable}.
  \]
  Note that none of the above implications can be reversed. It is well known
  that $S_{\omega}$ is $csf$-countable but not
  $snf$-countable. Moreover, any space without non-trivial convergent
  sequences is $snf$-countable, see Example~\ref{e11}.

  %\medskip
  Given a group $G$, let $e_G$ denote the neutral element of $G$. If no
  confusion occurs, we simply use $e$ instead of $e_G$ to denote the neutral
  element of $G$. Let $N: G \to \mathbb R$ be a function. We call $N$ a
  \emph{pre-norm} on $G$ if the following conditions are satisfied for all
  $x, y\in G$: (i) $N(e)=0$; (ii) $N(xy^{-1})\leq N(x)+N(y)$. If $G$ is a
  topological space and $N$ is continuous, then we say that $N$ is a
  {\it continuous pre-norm} on $G$.

  %\medskip
  Let $X$ be a non-empty Tychonoff space. Throughout this paper, $X^{-1}
  :=\{x^{-1}: x\in X\}$ and $-X: =\{-x: x\in X\}$, which are just two
  copies of $X$. For every $n\in\mathbb{N}$, $F_{n}(X)$ denotes the
  subspace of $F(X)$ that consists of all words of reduced length at
  most $n$ with respect to the free basis $X$. The subspace
  $A_{n}(X)$ is defined similarly. We always use $G(X)$ to denote
  topological groups $F(X)$ or $A(X)$, and $G_{n}(X)$ to $F_{n}(X)$ or
  $A_{n}(X)$ for each $n\in \N$. Therefore, any statement about $G(X)$
  applies to $F(X)$ and $A(X)$, and $G_{n}(X)$ applies to $F(X)$ and
  $A(X)$. Let $e$ be the neutral element of $F(X)$ (i.e., the empty
  word) and $0$ be that of $A(X)$. For every $n\in\N$ and an element
  $(x_{1}, x_{2}, \cdots, x_{n})$ of $(X\bigoplus X^{-1}\bigoplus\{e\})^{n}$
  we call $g=x_{1}x_{2}\cdots x_{n}$ a {\it form}. In the Abelian case,
  $x_{1}+x_{2}\cdots +x_{n}$ is also called a {\it form} for $(x_{1}, x_{2},
  \cdots, x_{n})\in(X\bigoplus -X\bigoplus\{0\})^{n}$. This word $g$ is
  called {\it reduced} if it does not contains $e$ or any pair of
  consecutive symbols of the form $xx^{-1}$ or $x^{-1}x$. It follows
  that if the word $g$ is reduced and non-empty, then it is different
  from the neutral element $e$ of $F(X)$. In particular, each element
  $g\in F(X)$ distinct from the neutral element can be uniquely written
  in the form $g=x_{1}^{r_{1}}x_{2}^{r_{2}}\cdots x_{n}^{r_{n}}$, where
  $n\geq 1$, $r_{i}\in\mathbb{Z}\setminus\{0\}$, $x_{i}\in X$, and
  $x_{i}\neq x_{i+1}$ for each $i=1, \cdots, n-1$, and the {\it support}
  of $g=x_{1}^{r_{1}} x_{2}^{r_{2}}\cdots x_{n}^{r_{n}}$ is defined as
  $\mbox{supp}(g) :=\{x_{1}, \cdots, x_{n}\}$. Given a subset $K$ of
  $F(X)$, we put $\mbox{supp}(K):=\bigcup_{g\in K}\mbox{supp}(g)$. Similar
  assertions (with the obvious changes for commutativity) are valid for
  $A(X)$. For every $n\in\mathbb{N}$, let $i_n: (X\bigoplus X^{-1}
  \bigoplus\{e\})^{n} \to F_n(X)$ be the natural mapping defined by
  $i_n(x_1, x_2, ... x_n)= x_1x_2...x_n$
  for each $(x_1, x_2, ... x_n) \in (X\bigoplus X^{-1} \bigoplus\{e\})^{n}$.
  We also use the same symbol in the Abelian case, that is, it means
  the natural mapping from $(X\bigoplus -X\bigoplus\{0\})^{n}$ onto
  $A_{n}(X)$. Clearly, each $i_n$ is a continuous mapping.

  %\newpage

  %%%%%%%%%%%%%%%%%%%%%%%%%%%%%%%%%%%%%%%%%%%%%%%%%%%%%%%%%%%%%%%%%%%%%%%
  \section{The $snf$-Countability of Free Topological Groups} \label{sec:snf}
  %%%%%%%%%%%%%%%%%%%%%%%%%%%%%%%%%%%%%%%%%%%%%%%%%%%%%%%%%%%%%%%%%%%%%%%

  In this section, we discuss the $snf$-countability of $G(X)$ and $G_n(X)$
  for a given topological space $X$. First, we provide some general
  characterizations for the $snf$-countability of $G(X)$. Then, we give
  some particular classes of spaces $X$ for which those characterizations
  hold. Finally, we characterizes the $snf$-countability of $F_4(X)$ for a
  specific class of spaces $X$, namely $k$-spaces with a $G_\delta$-diagonal.

  %\medskip
  The following theorem generalizes Corollary 4.14 in \cite{A1980}.

  \begin{theorem}\label{thm:freegeneral}
  For a space $X$, the following statements are equivalent:
  \begin{enumerate}
  \item[(i)] $G(X)$ is $snf$-countable.

  \item[(ii)] Each $G_{n}(X)$ contains no non-trivial convergent sequences.

  \item[(iii)] $G(X)$ contains no non-trivial convergent sequences.
  \end{enumerate}
  \end{theorem}

  \begin{proof}
  The implications (iii) $\Rightarrow$ (ii) and (iii) $\Rightarrow$ (i)
  are obvious.

  %\medskip
  (ii) $\Rightarrow$ (iii). Assume that each $G_{n}(X)$ contains no non-trivial
  convergent sequences. If $G(X)$ contains a non-trivial convergent sequence
  $S$, there is an $n\in \mathbb{N}$ such that $S\subseteq G_{n}(X)$. This
  implies that $G_{n}(X)$ contains a non-trivial convergent sequence,
  which is a contradiction.

  %\medskip
  (i) $\Rightarrow$ (iii). Let $G(X)$ be $snf$-countable with a countable
  $sn$-network $\{U_{n}: n\in\mathbb{N}\}$ at $e$, where $U_{n+1}
  \subseteq U_{n}$ for each $n\in\mathbb{N}$. Assume that $G(X)$ contains
  a non-trivial convergent sequence $\{x_{i}: i\in\mathbb{N}\}$ converging
  to $e$. Without loss of generality, we assume that $x_{i}\neq e$ for
  each $i\in \mathbb{N}$. Then there exists $2\leq n_{0}\in\mathbb{N}$ such
  that $\{x_{i}: i\in\mathbb{N}\}\subseteq G_{n_{0}}(X)$. We may further
  assume that $\{x_{i}: i\in\mathbb{N}\}\subseteq G_{n_{0}}(X)\setminus
  G_{n_{0}-1}(X)$.

  %\medskip
  We first consider the case of $F(X)$. For each $i\in \mathbb{N}$,
  let $x_{i}(1)\in X\cup X^{-1}$ be the first letter of $x_{i}$. Put
  $A:=\{i\in\mathbb{N}: x_{i}(1)\in X\}$ and $B:=\{i\in \mathbb{N}:
  x_{i}(1)\in X^{-1}\}$. Obviously, we have $A\cup B=\mathbb{N}$. Without
  loss of generality, we may assume that $|A|=\omega$. We further assume
  that $A=\N$. Take an arbitrary point $x\in X$. For each $m\in\N$,
  let $L_{m}:=\{x^{2^{m}n_0}x_ix^{-2^{m}n_0}: i\in \mathbb{N}\}$. Then
  $L_{m}\cap L_{n}=\emptyset$ for any $m\neq n$. Indeed, take arbitrary
  $m, n\in\mathbb{N}$ with $m<n$, and then pick any $g\in L_{m}$ and
  $h\in L_{n}$. Since each $x_{i}$ belongs to $G_{n_{0}}(X)\setminus
  G_{n_{0}-1}(X)$, we have
  \[
  \ell(g)> 2^{m+1}n_{0}-n_{0}\geq 2^{m}n_{0}+n_{0}\geq 2^{n+1}n_{0}+
  n_{0}\geq \ell (h).
  \]
  Hence $g\neq h,$ where $\ell(g)$ and $\ell(h)$ denote the lengths of
  $g$ and $h$, respectively. It is evident that each $\left\{x^{2^{m}n_0}
  x_ix^{-2^{m}n_0}: i\in \mathbb{N}\right\}$ converges to $e$. For each
  $n\in \mathbb{N}$, pick a point $y_{n}\in U_{n}\cap L_{n}$. Then, the
  sequence $\{y_n: n\in \mathbb{N}\}$ converges to $e$. However, since
  $|\{y_{n}: n\in \mathbb{N}\}\cap F_{m}(X)|<\omega$ for each $m\in
  \mathbb{N}$, by Corollary 7.4.3 in \cite{AT2008}, the set $\{y_{n}:
  n\in \mathbb{N}\}$ is closed and discrete in $F(X)$.
  A contradiction occurs.

  %\medskip
  Now, we consider the case of $A(X)$. Let $L_{m}:=\{mx_{i}: i\in
  \mathbb{N}\}.$ It is easy to see that $L_{m}\cap L_{n}=\emptyset$ for
  any $m\neq n$. Obviously, each $\{mx_i: i\in \mathbb{N}\}$
  converges to $e$. We can derive a contradiction by a proof similar to
  that for the case of $F(X)$.
  \end{proof}

  Next, we consider what conditions on a space $X$ can guarantee $G(X)$
  to be $snf$-countable. In the light of Theorem \ref{thm:freegeneral},
  one may conjecture that $G(X)$ is $snf$-countable if $X$ contains no
  non-trivial convergent sequences. Unfortunately, this is not true,
  since there is a countably compact and separable $cf$-space $X$
  such that $A(X)$ and $F(X)$ contain a non-trivial convergent sequence,
  refer to Theorem 3.5 in \cite{T2013}. Nevertheless,
  by applying Theorem \ref{thm:freegeneral} to some special classes of
  topological spaces, we obtain the following result.

  \begin{theorem} \label{thm:cf_snf_countable}
  For a space $X$, $G(X)$ is $snf$-countable if one of the
  following holds:
  \begin{enumerate}
  \item[(i)] $X$ is a cf- and $\mu$-space;
  \item[(ii)] Every countable discrete subset of $X$ is $C^{\ast}$-embedded.
  \end{enumerate}
  \end{theorem}

  \begin{proof}
  (i) It suffices to show that $G(X)$ is a $cf$-space. If $K\subseteq
  G(X)$ is compact, then $K\subseteq G_{n}(X)$ for some $n\in
  \mathbb{N}$. Let $Z:=\overline{\mbox{supp}(K)}$. Then $Z$ is a
  compact subset in $X$, as $X$ is a $\mu$-space. Hence $Z$ is finite.
  Since $K\subseteq i_{n}((Z\bigoplus Z^{-1}\bigoplus\{e\})^{n})$ (in
  Abelian case, $K\subseteq i_{n}((Z\bigoplus -Z\bigoplus\{0\})^{n})$),
  $K$ must be finite. Therefore, $G(X)$ is $snf$-countable.
  	
  %\medskip
  (ii) Since $X$ is a space in which every countable discrete subset
  is $C^{\ast}$-embedded, it follows from Porposition 2.4 in
  \cite{T2013} that $G(X)$ contains no non-trivial convergent sequences.
  Therefore, by Theorem \ref{thm:freegeneral}, $G(X)$ is $snf$-countable.
  \end{proof}

  A topological space $X$ is {\it extremely disconnected} if the closure
  of every open subset is open. Since every countable discrete subset
  of an extremely disconnected space $X$ is $C^{\ast}$-embedded, by
  Theorem \ref{thm:cf_snf_countable}, $G(X)$ is $snf$-countable. As a
  particular example, the Stone-\v{C}ech compactification $\beta D$ of
  any discrete space $D$ is extremely disconnected, and hence $G(\beta D)$
  is $snf$-countable.

  %\medskip
  If $X$ is a topological group, then we can characterize the
  $snf$-countability of $G(X)$ in terms of the property that $X$
  contains no non-trivial convergent sequences, as it is shown in
  the next result.

  \begin{theorem} \label{thm:non_trivial_cs}
  For a topological group $X$, $G(X)$ is $snf$-countable if and
  only if $X$ contains no non-trivial convergent sequences.
  \end{theorem}

  \begin{proof}
  The necessity is obvious by Theorem \ref{thm:freegeneral}. To show the
  sufficiency, suppose that $X$ contains no non-trivial convergent sequences.
  By Theorem \ref{thm:freegeneral}, it suffices to show that $G(X)$ contains
  no non-trivial convergent sequences. We only consider the case of $F(X)$,
  since the proof for the case of $A(X)$ is quite similar. Assume that
  $F(X)$ contains a non-trivial convergent sequence. It follows from
  \cite[Proposition 2.1]{T2013} that there are two sequences $\{x_n: n\in
  \mathbb{N}\}$ and $\{y_n: n\in \mathbb{N}\}$ in $X$ such that $\{x_{n}:
  n\in\N\}$ is infinite, $x_{n}\neq y_n$ for each $n\in \N$ and for every
  continuous pseudometric $d$ on $X$, and $d(x_n, y_n)\geq 1$ for at most
  finitely many $n\in \N$. Let $S := \{x_n^{-1}y_n: n\in\N\}$. Then $S$
  is a non-trivial sequence.

  %\medskip
  We claim that $S$ is convergent to $e$ in $X$. If not, there exists
  an open neighborhood $U$ of $e$ in $X$ and an infinite set $A\subseteq
  \mathbb{N}$ such that $x^{-1}_{n}y_{n}\not\in U$ for each $n\in A$.
  By Theorem 3.3.9 in \cite{AT2008}, there exists a continuous pre-norm
  $N$ on $X$ such that $\{g\in X: N(g)<1\}\subseteq U$.
  Define a continuous pseudometric $d$ on $X$ by $d(x, y)=N(x^{-1} y)$,
  for all $x, y\in X$. It follows that $d(x_{n}, y_{n})\geq 1$ for
  all $n\in A$. However, this is impossible, since by the construction
  of $\{x_n: n\in \mathbb{N}\}$ and $\{y_n: n\in \mathbb{N}\}$,
  $d(x_n, y_n)\geq 1$ holds only for at most finitely many $n\in A$.
  Thus, $S$ is a non-trivial sequence in $X$ converging to $e$. This
  contradicts with the fact that $X$ contains no non-trivial convergent
  sequences.
  \end{proof}

  To see how Theorem \ref{thm:non_trivial_cs} can be applied, we need
  to identify some classes of topological groups that contain no
  non-trivial convergence sequences.
  By a result of \cite{B2012}, every hereditarily normal topological
  group without a $G_{\delta}$-diagonal contains no non-trivial
  convergent sequences, and thus $G(X)$ is $snf$-countable for such
  a topological group $X$. On the other hand, every infinite compact
  group $X$ contains a non-trivial convergent sequence, as it contains
  a copy of the Cantor cube $\{0, 1\}^{w(X)}$, where
  $w(X)$ is the weight of $X$. Hence, $G(X)$ is not $snf$-countable
  for any infinite compact group $X$. Note that there are infinite
  pseudocompact topological groups containing no non-trivial convergent
  sequences, refer to \cite{Si1969} for the existence of such a
  topological group.

  %\medskip
  Next, we discuss the $snf$-countability of subspaces $F_n(X)$ of
  $F(X)$ for a space $X$. Recall that a subspace $Y$ of a space $X$ is
  said to be {\it P-embedded in $X$} if each continuous pseudometric
  on $Y$ admits a continuous extensions over $X$.

  \begin{theorem}\label{t11}
  Let $X$ be a space with a regular $G_{\delta}$-diagonal. If $F_{4}(X)$
  is $snf$-countable, then $X$ is either a $cf$-space or compact.
  \end{theorem}

  \begin{proof}
  Suppose that $X$ contains an infinite compact subset $C$. We show
  that $X$ must be compact. Since $X$ has a regular $G_{\delta}$-diagonal,
  then $C$ must be a metrizable subspace. Without loss
  of generality, we may assume that $C$ is a non-trivial convergent
  sequence with its limit point $x$. Since by a result in \cite{AB2006}
  any pseudocompact space with a regular $G_{\delta}$-diagonal is
  metrizable and compact, we only need to show that $X$ is pseudocompact.
  Assume that $X$ is not pseudocompact. Then $X$ contains an infinite
  and discrete sequence of open subsets $\{U_{n}: n\in \mathbb{N}\}$.
  Note that $\{\overline{U}_n: n\in \mathbb{N}\}$ is also discrete in
  $X$. Therefore, $\bigcup _{n\in \mathbb{N}} \overline{U}_n$ is
  closed in $X$. Since $C$ is compact, we may assume that $C\cap
  \left(\bigcup _{n\in \mathbb{N}}\overline{U}_n \right)=\emptyset$.
  It follows that $\{C\}\cup\{U_{n}: n\in \mathbb{N}\}$ is discrete
  in $X$. For each $n\in \mathbb{N}$, take a point $x_{n}\in U_{n}$.
  Then, $Y:=C\cup \{x_{n}: n\in \mathbb{N}\}$ is closed, $\sigma$-compact
  and $P$-embedded in $X$. By a result in \cite{U1991},
  the subgroup $F(Y, X)$ of $F(X)$ generated by $Y$ is topologically
  isomorphic to $F(Y)$. Obviously, $F(Y)$ is a $k_{\omega}$-space. We
  claim that $F_4(Y)$ contains a closed copy of $S_{\omega}$. For each
  $n\in \mathbb{N}$, put $C_n :=x_{n}Cx^{-1}x_{n}^{-1}$. Let $Z :=
  \bigcup_{n\in\mathbb{N}}C_n$. Then $Z \subseteq F(Y)$ is closed.
  Since $F(Y)$ is a $k_{\omega}$-space, $Z$ is also a
  $k_{\omega}$-subspace. Take an arbitrary infinite subset
  \[
  P:= \{x_{n_{m}}c_{m}x^{-1}x_{n_{m}}^{-1}: m\in\mathbb{N}, c_{m}\in
  C\setminus\{x\}\}.
  \]
  Then $P$ is a discrete closed subset of $Z$, since $Z$ is a
  $k_{\omega}$-space. Therefore, the claim is verified. It follows
  that $S_{\omega}$ must be $snf$-countable, which is a contradiction.
  \end{proof}

  The following theorem generalizes Theorem 4.9 in \cite{Y1998}.

  \begin{theorem} \label{thm:subspacefree}
  For a $k$-space $X$ with a regular $G_{\delta}$-diagonal, the
  following are equivalent:
  \begin{enumerate}
  \item[(i)] $F_{4}(X)$ is $snf$-countable.
  \item[(ii)] Each $F_{n}(X)$ is $snf$-countable.
  \item[(iii)] $F_{4}(X)$ is metrizable.
  \item[(iv)] Each $F_{n}(X)$ is metrizable.
  \item[(v)] $X$ is discrete or compact.
  \end{enumerate}
  \end{theorem}

  \begin{proof}
  Obviously, we have (ii) $\Rightarrow$ (i), (iv) $\Rightarrow$ (iii),
  (iv) $\Rightarrow$ (ii) and (iii) $\Rightarrow$ (i). It suffices to
  show that (i) $\Rightarrow$ (v) and (v) $\Rightarrow$ (iv).

  %\medskip
  (i) $\Rightarrow$ (v). Since $X$ has a regular $G_{\delta}$-diagonal,
  it follows from Theorem~\ref{t11} that either each compact subset
  of $X$ is finite or $X$ is compact and metrizable. Moreover, it is
  easy to check that a $k$-space is discrete if each compact subset
  is finite. Therefore, $X$ is discrete or compact.

  %\medskip
  (v) $\Rightarrow$ (iv). If $X$ is discrete, then $F(X)$ is discrete,
  hence each $F_{n}(X)$ is metrizable. If $X$ is compact, then $X$ is
  a compact metrizable space since a compact space with a
  $G_{\delta}$-diagonal is metrizable \cite{Gr}. Moreover, since each
  $i_{n}$ is continuous, each $F_{n}(X)$ is compact. Therefore, each
  $F_{n}(X)$ is metrizable.
  \end{proof}

  The next example shows that the condition that $X$ is ``a $k$-space"
  in Theorem \ref{thm:subspacefree} cannot be dropped.

  \begin{example}\label{e11}
  \emph{There is an infinite, non-discrete, $cf$-space $X$ with a regular
  $G_\delta$-diagonal such that $F_4(X)$ is $snf$-countable.} Let $\beta
  \N$ be the Stone-\v{C}ech compactification of $\N$ (equipped with
  the discrete topology). Take an arbitrary
  point $p\in \beta\mathbb{N} \setminus \mathbb{N}$, and consider the
  subspace $X:=\mathbb{N}\cup \{p\}$ of $\beta\N$. It is well known
  that $X$ is not compact. Indeed, $X$ is a non-discrete $cf$-space with a
  regular $G_\delta$-diagonal. However, as shown in \cite[Example 3.12]{LLL},
  $F(X)$ is $snf$-countable. Thus $F_4(X)$
  is also $snf$-countable.
  \end{example}

  By Theorem 4.12 in \cite{Y1998}, it is easy to see that ``$F_{4}(X)$'' in
  Theorem \ref{thm:subspacefree} cannot be replaced by ``$F_{3}(X)$''.
  However, we have the following result.

  \begin{theorem}\label{thm:derivedset}
  Let $X$ be a stratifiable $k$-space. If $G_{2}(X)$ is $snf$-countable,
  then $X'$ is compact.
  \end{theorem}

  \begin{proof}
  We only consider the case of $A_{2}(X)$, as the proof of the $F_{2}(X)$
  case is quite similar.

  %\medskip
  Suppose that $X'$ is not compact. Then, $X'$ is not countably compact
  since $X$ is stratifiable. Therefore, there is a closed, countable,
  infinite and discrete subset $\{x_n: n\in\mathbb{N}\}$ in $X'$. Since
  $X$ is paracompact, we can choose a discrete family $\{U_n: n\in\N\}$ of
  open subsets in $X$ such that $x_n \in U_n$ for each $n\in\mathbb{N}$.
  For each $n\in\mathbb{N}$, since $X$ is sequential and $\{x_n: n\in
  \mathbb{N}\}\subseteq X'$, we can choose a non-trivial sequence $\{x_n(m):
  m\in \mathbb{N}\}$ such that $\{x_n(m): m\in \mathbb{N}\}$ converges to
  $x_{n}$ and $\{x_n(m): m\in \mathbb{N}\}\subseteq U_n$. Let $Y_n:=\{x_n(m):
  m\in \N\} \cup \{x_n\}$, $C_n :=\{x_n(m)-x_n: m\in \N\}$ and
  $C :=\bigcup\{C_n: n\in\N\}\cup\{0\}$. Obviously, each sequence
  $\{x_n(m)-x_n: m \in \N\}$ converges to 0 in $A_{2}(X)$ and $\{Y_n:
  n\in\mathbb{N}\}$ is a discrete family in $X$.

  %\medskip
  We claim that the subspace $C$ is a copy of $S_{\omega}$. Indeed,
  the subspace $S$, defined by
  \[
  S:=\{(x_{n}(m), -x_{n}): m, n\in\mathbb{N}\}\cup\{(x_{n}, -x_{n}): n\in
  \mathbb{N}\},
  \]
  is closed in $(X\cup -X)\times(X\cup -X)$. Since $X$ is paracompact,
  it follows from Proposition 4.8 in \cite{Y1998} that $i_2$ is a closed
  map. Then $i_2 \uhr S$ is a quotient mapping. Therefore, $C\subseteq
  A_{2}(X)$ is homeomorphic to $S_{\omega}$, and this verifies the claim.
  Since $A_{2}(X)$ is $snf$-countable, then $C$ is $snf$-countable. This
  contradicts with the fact that $S_{\omega}$ is not $snf$-countable.
  \end{proof}

  Note that in general, the converse of Theorem \ref{thm:derivedset}
  does not hold. To see this, consider the space $S_{\omega}$. It is
  easy to check that $S_{\omega}$ is a stratifiable $k$-space whose
  set of non-isolated points is compact. However, neither $S_{\omega}$
  nor $G_2(S_{\omega})$ is $snf$-countable.

  %\newpage

  %%%%%%%%%%%%%%%%%%%%%%%%%%%%%%%%%%%%%%%%%%%%%%%%%%%%%%%%%%%
  \section{The $csf$-Countability of Free Topological Groups}
  %%%%%%%%%%%%%%%%%%%%%%%%%%%%%%%%%%%%%%%%%%%%%%%%%%%%%%%%%%%

  In this section, we discuss the $csf$-countability of $G(X)$ and
  $G_n(X)$ for a given space $X$. First of all, we have the following
  simple observation.

  \begin{proposition} \label{prop:csf}
  For a space $X$, $G(X)$ is $csf$-countable iff $G_n(X)$ is
  $csf$-countable for all $n \in \mathbb N$.
  \end{proposition}

  \begin{proof}
  It is clear that if $G(X)$ is $csf$-countable, then each $G_n(X)$
  is $csf$-countable.

  %\medskip
  Conversely, if each $G_n(X)$ is $csf$-countable with a $cs$-network
  $\mathscr{P}_n =\{P_{n, i}: i\in \N\}$ at $e$, then it is easy to
  check that $\mathscr{P}=\bigcup_{n\in\N} \mathscr{P}_n$ is a
  $cs$-network for $G(X)$ at $e$. Hence, if $G_n(X)$ is
  $csf$-countable for all $n \in \mathbb N$, then $G(X)$ is
  $csf$-countable.
  \end{proof}

  In the light of Proposition \ref{prop:csf}, one of our purposes in
  this section is to identify those classes of spaces $X$ for which
  the $csf$-countability of $G_n(X)$ at certain level $n\in \N$ will
  be adequate to guarantee the $csf$-countability of $G(X)$.

  \begin{theorem}\label{thm:level2_csf}
  For a paracompact, crowded, $k$- and $\aleph$-space $X$, the
  following are equivalent:
  \begin{enumerate}
  \item[(i)] $G(X)$ is $csf$-countable.
  \item[(ii)] $G_{2}(X)$ is $csf$-countable.
  \item[(iii)] $X$ is separable.
  \end{enumerate}
  \end{theorem}

  \begin{proof}
  Since (i) $\Rightarrow$ (ii) is trivial, we only need to show (ii)
  $\Rightarrow$ (iii) and (iii) $\Rightarrow$ (i). Further, we only
  consider $A(X)$, as the proof of the case of $F(X)$ is quite similar.

  %\medskip
  (ii) $\Rightarrow$ (iii). Suppose $X$ is not separable. Since $X$ is
  an $\aleph$-space, there is a closed, uncountable and discrete subset
  $\{x_\alpha: \alpha<\omega_1\}$ in $X$. Since $X$ is paracompact, we
  can choose a discrete family of open subsets $\{U_\alpha: \alpha<\omega_1\}$
  in $X$ such that $x_\alpha\in U_\alpha$ for each $\alpha<\omega_1$.
  For each $\alpha<\omega_1$, since $X$ is sequential, we can choose
  a non-trivial sequence $\{x_n(\alpha): n\in \mathbb{N}\}\subseteq
  U_\alpha$, convergent to $x_\alpha$. Put
  \[
  Y_\alpha:=\{x_n(\alpha): n\in \mathbb{N}\}\cup \{x_\alpha\},\  C_\alpha:
  =\{x_n(\alpha)-x_\alpha: n\in \mathbb{N}\}\
  \]
  and let $C:=\bigcup\{C_\alpha: \alpha<\omega_1\}\cup\{0\}$. Obviously,
  $\{x_n(\alpha)-x_\alpha: n\in \mathbb{N}\}$ is convergent to 0 in
  $A_{2}(X)$ and $\{Y_\alpha: \alpha<\omega_1\}$ is a discrete family in
  $X$. We claimed that $C$ is a copy of $S_{\omega_1}$. Indeed, let
  \[
  S:=\{(x_{n}(\alpha), -x_{\alpha}): n\in\mathbb{N}, \alpha<\omega_1\}
  \cup\{(x_{\alpha}, -x_{\alpha}): \alpha<\omega_1\}.
  \]
  Then $S$ is closed in $(X\cup -X)\times (X\cup -X)$. Since $X$ is
  paracompact, by Proposition 4.8 in \cite{Y1998}, $i_2$ is a closed map.
  It follows that $i_2\uhr S$ is a quotient mapping, and thus $C\subseteq
  A_{2}(X)$ is homeomorphic to $S_{\omega_1}$. Since $A_{2}(X)$ is
  $csf$-countable, then $C$ is $csf$-countable. However, $S_{\omega_1}$
  is not $csf$-countable, which is a contradiction.

  %\medskip
  (iii) $\Rightarrow$ (i). Let $X$ be separable. Then $X$ is an $\aleph_0$-space.
  By Theorem 4.1 in \cite{AOP1989}, $A(X)$ is an $\aleph_0$-space. Thus,
  $A(X)$ is $csf$-countable.
  \end{proof}

  The following theorem provides an affirmative answer to Question 3.9 in
  \cite{LLL}. Recall that a topological space $X$ is said to be
  \emph{La\v{s}nev} if it is the closed image of some metric space.

  \begin{theorem}\label{thm:level4_csf}
  Let $X$ be a non-discrete La\v{s}nev space. Then $F_{4}(X)$ is
  $csf$-countable if and only if $F(X)$ is an $\aleph_0$-space.
  \end{theorem}

  \begin{proof}
  The sufficiency is obvious. To show the necessity, let $F_{4}(X)$
  be $csf$-countable. Then $X$ contains no copy of $S_{\omega_{1}}$, and
  hence $X$ is an $\aleph$-space. By Theorem 4.1 in \cite{AOP1989}, it suffices
  to show that $X$ is an $\aleph_0$-space. Since $X$ is an $\aleph$-space,
  the proof will be completed if we can show that each closed and discrete
  subset of $X$ is at most countable.
  Assume that $X$ contains an uncountable, closed and discrete subset
  $D=\{d_{\alpha}: \alpha<\omega_{1}\}$. Since $X$ is a non-discrete
  La\v{s}nev space, there exists a non-trivial convergent sequence $\{x_n: n
  \in\mathbb{N}\}$ with the limit point $x$ in $X$ such that $D\bigoplus S$ is
  a closed copy of $X$, where $S:=\{x_{n}: n\in\mathbb{N}\}\cup\{x\}$. Since
  $F_{4}(D\bigoplus S)$ is a closed subspace of $F_{4}(X)$, $F_{4}(D\bigoplus
  S)$ is $csf$-countable.

  %\medskip
  Next, we shall show that $F_{4}(D\bigoplus S)$ contains a copy of $S_\omega$.
  To this end, for each $\alpha < \omega_1$, define $C_\alpha :=\{ d_\alpha x_n
  x^{-1}{d_\alpha}^{-1}: n\in\mathbb{N}\}$. Then $C_{\alpha}$ converges to $e$.
  Let $C :=\{e\}\cup \left(\bigcup_{\alpha<\omega_1} C_\alpha\right)$. Obviously,
  $C\subseteq F_{4}(D\bigoplus S)$, which implies that $C$ is $csf$-countable.
  Then $C$ has a countable $cs$-network $\{P_{n}: n\in\mathbb{N}\}$ at $e$.
  Put
  \[
  N_{1}:=\{n\in\mathbb{N}: |\{\alpha<\omega_{1}: P_{n}\cap C_{\alpha}
  \neq\emptyset\}|\leq\aleph_{0}\}
  \]
  and
  \[
  B:=\{\alpha<\omega_{1}: C_{\alpha}\cap P_{_{n}}\neq\emptyset\ \mbox{for
  some}\ n\in N_{1}\}.
  \]
  Clearly, $\mathbb{N}\setminus N_{1}$ is a countable infinite set and $B$
  is a countable set. It is easy to see that $\{P_{n}: n\in\mathbb{N}
  \setminus N_{1}\}$ is a countable $cs$-network at $e$ in
  \[
  D_{1}:=\{e\}\cup\{d_{\alpha}x_{n}x^{-1}d_{\alpha}^{-1}: \alpha\in\omega_1
  \setminus B, n\in\mathbb{N}\}.
  \]
  Moreover, for each $n\in\mathbb{N}\setminus N_{1}$, the set $P_{n}$
  intersects uncountably many $C_{\alpha}$. Inductively, we can find a
  countable infinite subset $R:=\{\alpha_{n}\in\omega_{1} \setminus B:
  n\in\mathbb{N}\setminus N_{1}\}$ of $\omega_{1}$ such that
  $\alpha_{n}\neq \alpha_{m}$ if $n\neq m$ and $P_{n}\cap C_{\alpha_{n}}
  \neq\emptyset$ for each $n\in\mathbb{N} \setminus N_{1}$. Let $Y:=\{e\}
  \cup\left(\bigcup_{\alpha \in R} C_\alpha\right)$. It is clear that
  $Y \subseteq F_{4}(D\bigoplus S)$. We claim that $Y$ is homeomorphic to
  $S_{\omega}$. Indeed, let $Z=\mbox{supp}(Y)$. Then, $Z$ is a countable,
  infinite, locally compact and closed subspace in $D\bigoplus S$.
  Hence, $F(Z)$ is a sequential space by Theorem 7.6.36 in \cite{AT2008}.
  Moreover, $F(Z)$ is a subspace of $F(D\bigoplus S)$, since
  $D\bigoplus S$ is metrizable and $Z$ is closed in $D\bigoplus S$.
  Assume that $Y$ is not a copy of $S_{\omega}$. Then there is a
  non-trivial convergent sequence $\{d_{\alpha_{n_{k}}} x_{m_{k}}x^{-1}
  d_{\alpha_{n_{k}}}^{-1}: k\in\N\}$ in $Y$. Since $Z$ is paracompact,
  the closure of $\mbox{supp}\{d_{\alpha_{n_{k}}} x_{m_{k}}x^{-1}
  d_{\alpha_{n_{k}}}^{-1}: k\in\mathbb{N}\}$ in $Z$ is compact. However,
  $\{d_{\alpha_{n_{k}}}: k\in\mathbb{N}\}$ is an infinite, closed and
  discrete subspace in $Z$, which is a contradiction. Therefore, the
  claim is verified.

  %\medskip
  Finally, pick an arbitrary point $a_{n} \in P_{n}\cap C_{\alpha_{n}}$
  for each $n\in\mathbb{N}\setminus N_{1}$. Since the family $\{P_{n}
  \cap Y: n\in \mathbb{N}\setminus N_{1}\}$ is a countable $cs$-network
  at $e$ in $Y$, $e$ is a cluster point of the set $\{a_{n}: n \in
  \mathbb{N}\setminus N_{1}\}$, which is a contradiction. Therefore, each
  closed and discrete subset of $X$ is at most countable.
  \end{proof}

  \begin{corollary}\label{cor:level4_csf}
  For a non-discrete metrizable space $X$, the following are equivalent:
  \begin{enumerate}
  \item[(i)] $X$ is separable.
  \item[(ii)] $F(X)$ is an $\aleph_0$-space.
  \item[(iii)] $F(X)$ is $csf$-countable.
  \item[(iv)] $F_{4}(X)$ is $csf$-countable.
  \end{enumerate}
  \end{corollary}

  \begin{proof}
  By Theorem 4.1 in \cite{AOP1989}, we have (i) $\Leftrightarrow$ (ii).
  Moreover, (ii) $\Rightarrow$ (iii) and (iii) $\Rightarrow$ (iv) are obvious.
  Finally, (iv) $\Rightarrow$ (ii) follows from Theorem~\ref{thm:level4_csf}.
  \end{proof}

  \begin{remark}
  (i) \emph{Theorem \ref{thm:level4_csf} does't hold for the Abelian case},
  refer to Theorem 4.5 in \cite{Y1998}.

  %\medskip
  (ii) \emph{In Corollary \ref{cor:level4_csf}, $F_{4}(X)$ cannot be replaced
  by $F_{3}(X)$.} Let $X:=D\bigoplus S$, where $D$ is an uncountable
  discrete space and $S$ is a non-trivial convergent sequence with its
  limit. Then, $F_{3}(X)$ is metrizable by Theorem 4.12 in \cite{Y1998},
  hence it is $csf$-countable. However, $X$ is not separable.

  %\medskip
  (iii) \emph{The conclusion of Corollary \ref{cor:level4_csf} does not hold
  for the $snf$-countability.} Consider $\R$ with the usual
  Euclidean topology. By Corollary \ref{cor:level4_csf}, $F(\mathbb{R})$
  is an $\aleph_0$-space. However, it follows from Theorem
  \ref{thm:subspacefree} that $F_{4}(\mathbb{R})$ is not $snf$-countable.
  \end{remark}

  Next, we shall consider the question when $G_3(X)$ is $csf$-countable
  for a given space $X$. First, we recall an important lemma in \cite{S2005}.

  \begin{lemma}[\cite{S2005}] \label{lcs}
  If $\mathscr{P}$ is a countable $cs^{\ast}$-network at $x\in X$, then
  the family
  \[
  \left\{\bigcup\mathscr{F}: \mathscr{F}\subseteq \mathscr{P},
  \mathscr{F}\ \mbox{is finite}\right\}
  \]
  is a countable $cs$-network at $x$.
  \end{lemma}

  \begin{theorem} \label{thm:level3_csf}
  Let $X$ be a sequential and $\mu$-space. If $X'$ has a countable
  $cs^*$-network in $X$, then $F_{3}(X)$ is $csf$-countable.
  \end{theorem}

  \begin{proof}
  Let $\mathscr{P}$ be a countable $cs^*$-network for $X'$ in $X$. By
  Lemma~\ref{lcs}, $X$ is $csf$-countable. Hence, $(X\bigoplus X^{-1}
  \bigoplus\{e\})^n$ is $csf$-countable for each $n\in\N$. Moreover,
  since $F_3(X)\setminus F_1(X)$
  is open in $F_3(X)$ and homeomorphic to a subspace of $(X\bigoplus
  X^{-1}\bigoplus\{e\})^3$,
  $F_3(X)$ is $csf$-countable at each point of $F_3(X) \setminus F_1(X)$.
  In what follows, we shall show that $F_3(X)$ is also $csf$-countable at
  each point of $F_1(X)$. To this end, take an arbitrary point $g\in
  F_1(X)$. By Lemma~\ref{lcs}, we are done if we can prove that $F_3(X)$
  has a countable $cs^*$-network at $g$. We divide the proof in three
  cases.

  %\medskip
  {\bf Case 1.} \emph{The point $g$ is isolated in $X\cup X^{-1}$}.

  %\medskip
  In this case, let
  \[
  \mathscr{B}(g) :=\{gP_1^{\varepsilon_1}P_2^{\varepsilon_2},
  P_1^{\varepsilon_1} P_2^{\varepsilon_2} g: P_1, P_2\in
  \mathscr{P}\cup\{\{g\}\}, \varepsilon_1, \varepsilon_2\in\{1, -1\},
  \varepsilon_1\neq\varepsilon_2\}.
  \]
  Obviously, $|\mathscr{B}(g)|\leq \omega$. We verify that
  $\mathscr{B}(g)$ is a $cs^*$-network for $F_{3}(X)$ at $g$. Take an
  arbitrary sequence $\{x_n: n\in\N\}$ converging to $g$ in $F_{3}(X)$
  and an open neighborhood $U$ of $g$ in $F_{3}(X)$. Without loss of
  generality, we may assume that $\{x_n: n \in\N\}$ is a non-trivial
  convergent sequence such that that $x_{n}\neq x_{m}$ whenever $n\neq m$.
  By Theorem 1.5 in \cite{AOP1989}, $\mbox{supp}(\{x_n: n\in\N\}\cup\{g\})$
  is bounded in $X$. Since $X$ is a $\mu$-space,
  $\overline{\mbox{supp}(\{x_n: n\in\mathbb{N}\}\cup\{g\}})$ is compact
  in $X$.  Let
  \[
  Z:=\left(\overline{\mbox{supp}\{x_n: n\in\mathbb{N}\}\cup\{g\}}\right)
  \cup \left(\overline{\mbox{supp}\{x_n: n\in\mathbb{N}\}\cup\{g\}}
  \right)^{-1}.
  \]
  Then $Z$ is sequentially compact. For each $n\in\mathbb{N}$, pick a
  point $y_n\in Z^3\cap i_3^{-1}(x_n)$. Then it follows from the
  sequential compactness of $X$ that $\{y_n: n\in \N\}$ has a
  subsequence $\{(z_i(1), z_i(2), z_i(3)): i\in \N\}$ converging to
  $z=(z_1, z_2, z_3)$ for some $z\in Z^3$. Since $x_{n} \neq x_{m}$ for
  any $n\neq m$, the sequence $\{y_n: n\in \N\}$ is a non-trivial sequence.
  Hence, we have that $z\in i_3^{-1}(g)$, $z_i\in (X'\cup\{g\})\cup(X'
  \cup \{g\})^{-1}$ for each $i\in\{1, 2, 3\}$, $|\{i: i\leq 3, z_{i}\neq
  g\}|$ is an even number, $z_1=g$ or $z_{3}=g$, and $z_{2}\in X'\cup
  (X')^{-1}$. Then $z\in\{(g, z_{2}, z_{2}^{-1}), (z_{2}^{-1}, z_{2},
  g)\}$. Without loss of generality, we assume that $z_{1}=g, z_{2}\in
  X'$ and $z=(g, z_{2}, z_{2}^{-1})$. Pick an open neighborhood $V$ of
  $z_2$ in $X$ such that $\{g\}\times V\times V^{-1}\subseteq i_3^{-1}(U)$.
  Since $g$ is isolated in $X\cup X^{-1}$ and the sequence $\{z_{i}(1):
  i\in \mathbb{N}\}$ converges to $g$ in $X\cup X^{-1}$, $z_{i}(1)=g$
  for all but finitely many $i$. Then, since $\mathscr{P}$ is a
  $cs^*$-network for $X'$, it is easy to see that there exist $P_1,
  P_2\in \mathscr{P}$ such that $z_2 \in P_1 \subseteq V$, $z_2^{-1}\in
  P_2^{-1}\subseteq V^{-1}$ and $z\in \{g\}\times P_1\times P_2^{-1}$
  contains a subsequence of $\{(z_i(1), z_i(2), z_i(3)): i\in \N\}$.
  Hence, $g\in gP_1P_2^{-1} \subseteq gVV^{-1}\subseteq U$
  and $gP_1P_2^{-1}$ contains a subsequence of $\{x_n: n\in\N\}$.

  %\medskip
  {\bf Case 2.} \emph{The point $g$ is non-isolated in $X\cup X^{-1}$}.

  %\medskip
  Without loss of generality, we assume that $g\in X.$ In this case,
  let
  \[
  \mathscr{B}(g):=\left\{P_1^{\varepsilon_1} P_2^{\varepsilon_2}
  P_3^{\varepsilon_3}: P_1, P_2, P_3\in \mathscr{P} \cup\{\{e\}\},
  \varepsilon_1, \varepsilon_2, \varepsilon_3\in\{1, -1\},
  \varepsilon_1 + \varepsilon_2 + \varepsilon_3=1\right\}.
  \]
  Obviously, $|\mathscr{B}(g)|\leq \omega$. We verify that $\mathscr{B}
  (g)$ is a $cs^*$-network for $F_3(X)$ at $g$. Take an arbitrary
  sequence $\{x_n: n\in\N\}$ converging to $g$ in $F_3(X)$ and an open
  neighborhood $U$ of $g$ in $F_3(X)$. Next we shall show that there
  exists a $B\in\mathscr{B}(g)$ such that $g\in P\subseteq U$ and $P$
  contains a subsequence of $\{x_n: n\in\N\}$. Without loss of generality,
  we assume that $\{x_n: n \in\N\}$ is a non-trivial convergent
  sequence such that that $x_{n}\neq x_{m}$ whenever $n\neq m$.

  %\medskip
  {\bf Subcase 2.1.} \emph{The sequence $\{x_n: n\in\N\}$ contains a
  subsequence $\{x_{n_i}: i \in \N\}$ which is contained in $X$.}

  %\medskip
  Since $\mathscr{P}$ is a $cs^*$-network for $X'$ in $X$, there exists
  a $P\in \mathscr{P}$ such that $P$ contains a subsequence of
  $\{x_{n_i}: i \in \N\}$ and $g\in P\subseteq U$. Hence, $P=P\{e\}\{e\}
  \in\mathscr{B}(g)$ contains a subsequence of $\{x_{n_i}: i\in\N\}$
  and $g\in P\subseteq U$.

  %\medskip
  {\bf Subcase 2.2.} \emph{The sequence $\{x_n: n\in\N\}$ does not contain
  any subsequence which is contained in $X$.}

  %\medskip
  Without loss of generality,
  we may assume that $\{x_n: n\in\N\}\subseteq F_3(X)\setminus F_2(X)$,
  since $(F_2(X)\setminus F_1(X))\cup \{e\}$ is clopen in $F_3(X)$.
  Similar to the proof of Case 1, let
  \[
  Z:=\left(\overline{\mbox{supp}\{x_n: n\in\mathbb{N}\}\cup\{g\}}\right)
  \cup \left(\overline{\mbox{supp}\{x_n: n\in\mathbb{N}\}\cup\{g\}}
  \right)^{-1}.
  \]
  Then $Z^3$ is sequentially compact. For each $n\in\mathbb{N}$, pick
  a point $y_n\in Z^3\cap i_3^{-1}(x_n)$. Then $\{y_n: n\in \N\}$ has a
  subsequence $\{(z_i(1), z_i(2), z_i(3))\}_{i\in \N}$ converging to
  $z=(z_1, z_2, z_3)$ for some $z\in Z^3$. Since $g$ is assumed to be
  non-isolated in $X\cup X^{-1}$, we have
  \[
  (X^{\prime}\cup \{g\})\cup
  (X^{\prime}\cup \{g\})^{-1}=X^{\prime}\cup (X^{\prime})^{-1}.
  \]
  Moreover, since $x_{n}\neq x_{m}$ for any $n\neq m$, the sequence
  $\{y_n: n\in \N\}$ is a non-trivial sequence. Hence, we have
  $z\in i_3^{-1}(g)$, $z_{i}\in X^{\prime}\cup (X^{\prime})^{-1}$
  for each $i\in\{1, 2, 3\}$, $z_{1}=g$ or $z_{3}=g$, and $z_{2}\in
  X^{\prime}\cup (X^{\prime})^{-1}$.
  Then $z\in\{(g, z_{2}, z_{2}^{-1}), (z_{2}^{-1}, z_{2}, g)\}$. Without
  loss of generality, we assume that $z_{1}=g, z_{2}\in X'$ and
  $z=(g, z_{2}, z_{2}^{-1})$. Pick open neighborhoods $V_{1}, V_{2}$ of
  $g$ and $z_{2}$ in $X$, respectively, such that
  $V_1\times V_2\times V_2^{-1} \subseteq i_3^{-1}(U)$.
  Since $\mathscr{P}$ is a $cs^*$-network for $X'$ in $X$, it is easy
  to see that there exist $P_1, P_2, P_3\in \mathscr{P}$ such that
  $g\in P_{1}\subseteq V_{1}$, $z_{2}\in P_{2}\subseteq V_{2}$,
  $z_{2}^{-1}\in P_{3}^{-1}\subseteq V_{2}^{-1}$
  and $z\in P_1\times P_2\times P_3^{-1}$ contains a subsequence
  of $\{(z_i(1), z_i(2), z_i(3)): i\in\N\}$. Hence,
  $g\in P_1P_2P_3^{-1}\subseteq V_1 V_2 V_2^{-1}\subseteq U$
  and $P_1P_2P_3^{-1}$ contains a subsequence of $\{x_n: n\in\N\}$.

  %\medskip
  {\bf Case 3:} $g=e$.

  %\medskip
  In this case, let
  \[
  \mathscr{B}(e):=\left\{P_1^{\varepsilon_1}P_2^{\varepsilon_2}: P_1, P_2 \in \mathscr{P}\cup\{\{e\}\}, \varepsilon_1, \varepsilon_2, \in \{1, -1\},
  \varepsilon_1 + \varepsilon_2=0\right\}.
  \]
  Obviously, $|\mathscr{B}(g)|\leq \omega$. We verify that $\mathscr{B}(g)$
  is a $cs^*$-network for $F_{3}(X)$ at $e$. Take an arbitrary sequence
   $\{x_n: n\in\N\}$ converging to $e$ and an open neighborhood $U$ of
  $e$ in $F_{3}(X)$. Without loss of generality, we assume that $\{x_n:
  n \in\N\}$ is a non-trivial convergent sequence such that $x_n\neq x_m$
  whenever $n\neq m$. Since $(F_{2}(X)\setminus F_{1}(X))\cup \{e\}$ is
  clopen in $F_3(X)$, we assume that
  \[
  \{x_n: n\in\N\}\subseteq F_2(X)\setminus F_{1}(X)).
  \]
  By an argument similar to that in Case 2, we can show that there exist
  $P_1, P_2\in\mathscr{B}(e)$ such that $e\in P_1^{\varepsilon_1}
  P_2^{\varepsilon_2}\subseteq U$ and $P_1^{\varepsilon_1} P_2^{\varepsilon_2}$
  contains a subsequence of $\{x_n: n\in\N\}$.
  \end{proof}

  As shown by the following result, for the Abelian case, the conclusion
  of Theorem \ref{thm:level3_csf} can be strengthened significantly.

  \begin{theorem} \label{thm:abelian_csf}
  Let $X$ is be a sequential and $\mu$-space. If $X'$ has a countable
  $cs^*$-network in $X$, then $A(X)$ is csf-countable.
  \end{theorem}

  \begin{proof}
  Since $A(X)$ is a topological group, we only need to prove that
  $A(X)$ has a countable $cs$-network at $0$. Let $\mathscr{P}$
  be a countable $cs^*$-network for $X'$ and
  \[
  \mathscr{B}(0) :=\left\{\sum_{i=1}^{2k}\varepsilon_i P_{i}:
  P_{i}\in \mathscr{P}\cup\{\{0\}\}, \varepsilon_i\in\{1, -1\},
  i\leq 2k, \sum_{j=1}^{2k}\varepsilon_j=0\right\}.
  \]
  Obviously, $|\mathscr{B}(0)|\leq \omega$.  Next, we verify
  that $\mathscr{B}(0)$ is a $cs^*$-network for $A(X)$ at $0$.

  %\medskip
  Take an arbitrary sequence $\{x_n: n\in\N\}$ converging to $0$
  in $A(X)$ and an arbitrary open neighborhood $U$ of $0$
  in $A(X)$. Without loss of generality, we may assume that $\{x_n:
  n \in \N \}$ is a non-trivial convergent sequence. Further, we
  may assume that $x_n \neq x_m$ for any $n\neq m$. Obviously,
  we have $\{x_n: n\in \mathbb{N}\}\cup\{0\}\subseteq A_l(X)$
  for some $l\in \mathbb{N}$. Then there exists some $m\leq l$ such
  that $A_m \setminus A_{m-1}$ contains a subsequence of $\{x_n: n
  \in\N\}$. Therefore, we may assume that $\{x_n: n\in\mathbb{N}\}$
  is contained in $A_m \setminus A_{m-1}$. Let
  \[
  Z :=\left(\overline{\mbox{supp}\{x_n: n\in\mathbb{N}\}\cup\{0\}}
  \right)\cup \left(-\ \overline{\mbox{supp}\{x_n: n\in\mathbb{N}\}
  \cup\{0\}}\right).
  \]
  By the proof of Theorem \ref{thm:level3_csf}, $Z^m$ is sequentially
  compact. For each $n\in\N$, pick a point $y_n\in Z^m\cap
  i_m^{-1}(x_n)$. Then $\{y_n: n\in\N\}$ has a subsequence $\{(z_i(1),
  z_i(2), ... z_i(m)): i\in\N\}$ converging to $z=(z_1, z_2, ...
  z_m)$ for some $z\in Z^m$. Obviously, $z\in i_m^{-1}(0)$ and
  $z_j\in X'\cup (-X)'$ for each $j\leq m$. Pick an open subset
  $V_j$ in $X \bigoplus -X \bigoplus\{0\}$ for each $j\leq m$ such that
  \[
  V_1\times V_2\times \cdots \times V_m\subseteq i_m^{-1}(U).
  \]
  For each $z_j\in X' \cup (-X)'$, choose a $P_j\in \mathscr{P}$
  inductively such that $z_j\in {\varepsilon_j} P_j\subseteq V_j$,
  ${\varepsilon_j}P_j$ contains a subsequence $\{z_{i_k}(j):
  k\in\N\}$ of $\{z_i(j): i\in\N\}$ and $\{(z_{i_k}(1),
  z_{i_k}(2), ... z_{i_k}(m)): k\in\N\}$ is a subsequence of
  $\{(z_i(1), z_i(2), ... z_i(m)): i\in\N\}$, where each
  $\varepsilon_j = 1$ or $-1$. Let
  \[
  B := \varepsilon_1 P_1+\varepsilon_2 P_2+\cdots +
  \varepsilon_{2k} P_{2k}.
  \]
  Then $0\in B\subseteq U$, which verifies that $\mathscr{B}(0)$
  is a $cs^*$-network at $0$. By Lemma~\ref{lcs}, $A(X)$ is $csf$-countable.
  \end{proof}

  \begin{corollary}
  For a La\v{s}nev space $X$, if $G_2(X)$ is $csf$-countable, then so is
  $G_{3}(X)$.
  \end{corollary}

  \begin{proof}
  Since $X$ is a La\v{s}nev space, it is sequential. By a proof analogous
  to that of the implication (ii) $\Rightarrow$ (iii) in Theorem
  \ref{thm:level2_csf}, one can show that $X'$ is a separable subspace
  of $X$. Hence it follows from \cite{JY} that $X$ is a paracompact
  $\aleph$-space, since $G_2(X)$ is $csf$-countable and $X$ is a La\v{s}nev
  space. Since $X'$ is closed and separable in $X$, $X'$ is Lindel\"{o}f,
  which shows that $X'$ is an $\aleph_{0}$-subspace in $X$. Hence $X'$
  has a countable $cs^{\ast}$-network in $X$. Therefore, $F_3(X)$ and
  $A_3(X)$ are $csf$-countable by Theorems \ref{thm:level3_csf} and
  \ref{thm:abelian_csf}, respectively.
  \end{proof}

  \begin{remark}
  (i) Let $X:=D\bigoplus K$, where $D$ is an uncountable discrete space
  and $K$ is a compact metric space. Then $A(X)$ is $csf$-countable. However,
  $F_{3}(X)$ and each $A_{n}(X)$ are first-countable by Theorem 4.5 and
  Theorem 4.12 in \cite{Y1998}, and $F_{4}(X)$ is not $csf$-countable by
  Corollary \ref{cor:level4_csf}.

  %\medskip
  (ii) \emph{In general, the converse of Theorem \ref{thm:level3_csf} or
  Theorem \ref{thm:abelian_csf} does not hold}. Indeed, let $X$ be an
  uncountable pseudocompact topological group containing no nontrivial
  convergent sequences, as given in \cite{Si1969}. By Theorem
  \ref{thm:non_trivial_cs}, $G(X)$ is $snf$-countable, and hence $G(X)$ is
  $csf$-countable. Since $X$ is an uncountable pseudocompact topological
  group, $X'$ is uncountable. Then $X'$ must not have a countable
  $cs^*$-network in $X$. Indeed, assume on the contrary that $X'$ has
  a countable $cs^*$-network in $X$. Then $X'=X$ is an $\aleph_0$-space,
  hence it is submetrizable. Since a pseudocompact submetrizable space
  is metrizable, it follows that $X$ is metrizable, which is a contradiction
  with the assumption.
  \end{remark}

  %%%%%%%%%%%%%%%%%%%%%%%%%%%%%%
  \section{Open Questions}
  %%%%%%%%%%%%%%%%%%%%%%%%%%%%%%

  We conclude this paper by posing some open questions. Our first open
  question concerns about the Abelian case of Theorem \ref{thm:subspacefree}.
  Theorem \ref{thm:subspacefree} establishes relationships among the
  $snf$-countability and the metrizability of $F_4(X)$ and each $F_n(X)$,
  as well as some properties of $X$ for a fairly large class of topological
  spaces. It is nice to know whether a similar result on $A_4(X)$ and
  $A_n(X)$ holds for the same class of spaces. Thus, the following question
  is of interests.

  \begin{question} \label{ques:abelian_sub_snf}
  Let $X$ be a $k$-space with a regular $G_{\delta}$-diagonal. If $A_{4}
  (X)$ is $snf$-countable, must every $A_{n}(X)$ be $snf$-countable?
  \end{question}

  In the light of Theorem 4.5 and Theorem 4.12 in \cite{Y1998}, Theorem
  \ref{thm:subspacefree} and Theorem \ref{thm:level4_csf}, it is natural
  to pose the following two questions.

  \begin{question} \label{ques:g2_to_g3}
  Let $X$ be a $k$-space with a regular $G_{\delta}$-diagonal.
  If $G_{2}(X)$ is $snf$-countable, must $G_{3}(X)$ be
  $snf$-countable?
  \end{question}

  \begin{question} \label{ques:level4_abelian}
  Let $X$ be a non-discrete La\v{s}nev space. If $A_{4}(X)$ is $csf$-countable,
  must each $A_{n}(X)$ be $csf$-countable?
  \end{question}

  Note that the answers to Question \ref{ques:g2_to_g3} and Question
  \ref{ques:level4_abelian} are not known even when $X$ is a metrizable space.
  Our last open question concerns about how to characterize the $csf$-countability
  of $F(X)$ and $A(X)$ in term of properties of $X$.

  \begin{question}
  Let $X$ be a space. Is there a topological property $\mathscr {A}$ of $X$
  which characterizes the $csf$-countability of $F(X)$ or $A(X)$?
  \end{question}

  %\bigskip

  %%%%%%%%%%%%%%%%%%%%%%%%%%%%

  \end{document}